\documentclass[english,12pt,a4paper,twoside]{smfart}
\usepackage{amsfonts,smfthm}
\usepackage{amsmath}
\usepackage{stmaryrd}
\usepackage{amssymb}
\usepackage{graphics}
\usepackage{babel}
\usepackage{enumerate}
\usepackage{euscript}

\usepackage{newlfont}
\usepackage{latexsym}
\usepackage{graphicx}
\usepackage{a4wide}

\usepackage{cite}        
\usepackage{url}         

\title{On the existence of zero-sum subsequences \\ of distinct lengths}
\author{Benjamin Girard}
\address{Graduate Center's Department of Mathematics\\ 
The City University of New York\\
365 Fifth Avenue\\
New York City, NY 10016-4309\\
USA.}
\email{bgirard@gc.cuny.edu}

\theoremstyle{plain}
\newtheorem{theorem}{Theorem}[section]

\newtheorem{corollary}[theorem]{Corollary}
\newtheorem{conjecture}{Conjecture}

\theoremstyle{definition}

\theoremstyle{remark}

\def\cnp[#1,#2]{\begin{pmatrix} #1 \\#2 \end{pmatrix}}
\begin{document}
\maketitle \setcounter{page}{1} \vspace{0.0cm}

\begin{abstract} 
In this paper, we obtain a characterization of short normal sequences over a finite Abelian $p$-group, thus answering positively a conjecture of Gao for a variety of such groups. 
Our main result is deduced from a theorem of Alon, Friedland and Kalai, originally proved so as to study the existence of regular subgraphs in almost regular graphs. 
In the special case of elementary $p$-groups, Gao's conjecture is solved using Alon's Combinatorial Nullstellensatz. 
To conclude, we show that, assuming every integer satisfies Property B, this conjecture holds in the case of finite Abelian groups of rank two. 
\end{abstract}

\vspace{-0.7cm}
\section{Introduction}
\label{Intro}
Let $\mathcal{P}$ be the set of prime numbers and let $G$ be a finite Abelian group, written additively. By $\exp(G)$ we denote the exponent of $G$. If $G$ is cyclic of order $n$, it will be denoted by $C_n$. In the general case, we can decompose $G$ as a direct product of cyclic groups $C_{n_1} \oplus \dots \oplus C_{n_r}$ where $1 < n_1 \text{ }
|\text{ } \dots \text{ }|\text{ } n_r \in \mathbb{N}$. For each $g$ in $G$, we denote by $\text{ord}(g)$ its order in $G$, and by $\left\langle g\right\rangle$ the subgroup it generates. 

\medskip
By a \em sequence \em over $G$ of \em length \em $\ell$, we mean a finite sequence of $\ell$ elements from $G$, where repetitions are allowed and the order of elements is disregarded. We use multiplicative notation for sequences. Let
$$S = g_1 \cdot \hdots \cdot g_{\ell} = \prod_{g \in G}g^{\mathsf{v}_g(S)}$$
be a sequence over $G$, where, for all $g \in G$, $\mathsf{v}_g(S) \in \mathbb{N}$ is called the \em multiplicity \em of $g$ in $S$. We call $\text{Supp}(S)=\{g \in G \mid \mathsf{v}_g(S)>0\}$ the \em support \em of $S$, and $\sigma(S)=\sum^{\ell}_{i=1}g_i=\sum_{g \in G}g\mathsf{v}_g(S)$ the \em sum \em of $S$. In addition, we say that $s \in G$ is a \em subsum \em of $S$ when
$$s=\displaystyle\sum_{i \in I} g_i \text{ for some }\emptyset \varsubsetneq I \subseteq \{1,\dots,\ell\}.$$

\noindent If $0$ is not a subsum of $S$, we say that $S$ is a \em zero-sumfree sequence\em. If $\sigma(S)=0,$ then $S$ is said to be a \em zero-sum sequence\em. 
If, moreover, one has $\sigma(T) \neq 0$ for all proper subsequences $T \mid S$, then $S$ is called a \em minimal zero-sum sequence\em.

\medskip
By $\mathsf{D}(G)$ we denote the smallest integer $t \in \mathbb{N}^*$ such that every sequence $S$ over $G$ of length $|S| \geq t$ contains a non-empty zero-sum subsequence. The number $\mathsf{D}(G)$ is called the \em Davenport constant \em of the group $G$. Even though its definition is purely combinatorial, the invariant $\mathsf{D}(G)$ found many applications in number theory (see for instance the book \cite{GeroKoch05} which presents the various aspects of non-unique factorization theory and \cite{GeroRuzsa09} for a recent survey). Thus, many direct and inverse problems related to $\mathsf{D}(G)$ have been studied during last decades, and even if numerous results were proved (see \cite[Chapter $5$]{GeroKoch05} and \cite{GaoGero06} for a survey), its exact value is known for very special types of groups only. 

\medskip
A sequence $S$ over a finite Abelian group $G$ is said to be \em dispersive \em if it contains two non-empty zero-sum subsequences $S_1$ and $S_2$ of distinct length, and \em non-dispersive \em otherwise, that is, when all non-empty zero-sum subsequences of $S$ have same length. One can readily notice, using the very definition of the Davenport constant, that every sequence $S$ over $G$ with $\left|S\right| \geq 2\mathsf{D}(G)$ is dispersive, since $S$ has to contain at least two disjoint non-empty zero-sum subsequences. So, one can ask for the smallest integer $t \in \mathbb{N}^*$ such that every sequence $S$ over $G$ with $|S| \geq t$ is dispersive. The associated inverse problem is then to make explicit the structure of non-dispersive sequences over a finite Abelian group. Concerning this problem, Gao, Hamidoune and Wang recently proved \cite{GaoHami09} that every non-dispersive sequence of $n$ elements in $C_n$ has at most two distinct values, solving a conjecture of Graham reported in a paper of Erd\H{o}s and Szemer\'{e}di \cite{ErdosSzemeredi76}. This result was then generalized by Grynkiewicz in \cite{Grynkie09}.

\medskip
In this article, we study a still widely open conjecture, proposed by Gao in \cite{Gao06}, on the structure of the so-called \em normal sequences \em over a finite Abelian group $G$. A sequence $S$ over $G$ with $\left|S\right| \geq \mathsf{D}(G)$ is said to be \em normal \em if all its zero-sum subsequences $S'$ satisfy $\left|S'\right| \leq \left|S\right|-\mathsf{D}(G) +1$. Gao's conjecture is the following (see also \cite[Conjecture $4.9$]{GaoGero06}).

\begin{conjecture}\label{conjecture Gao} Let $G \simeq C_{n_1} \oplus \dots \oplus C_{n_r},$ with $1 < n_1 \text{ } |\text{ } \dots \text{ }|\text{ } n_r \in \mathbb{N}$, be a finite Abelian group. Let also $S$ be a normal sequence over $G$ of length $\left|S\right|=\mathsf{D}(G)+i-1$, where $i \in \llbracket 1,n_1-1 \rrbracket$. Then $S$ is of the form $S=0^iT$, where $T$ is a zero-sumfree sequence.
\end{conjecture}

This notion of a normal sequence, first introduced in \cite{Gao06}, happens to be crucial in the characterization of sequences $S$ over $G$ with $\left|S\right|=\mathsf{D}(G)+\left|G\right|-2$ and which do not contain any zero-sum subsequence $S'$ satisfying $\left|S'\right|=\left|G\right|$ (see \cite[Theorem $1.7$]{Gao06}). The following two theorems, due to Gao (see \cite[Theorems $1.5$ and $1.6$]{Gao06}), are the only results known concerning the structure of normal sequences over a finite Abelian group. Before stating these two results on Conjecture \ref{conjecture Gao}, we recall that an integer $n \geq 2$ is said to satisfy \em Property B \em if every minimal zero-sum sequence over $C^2_n$ with $|S|=2n-1$ contains some element repeated $n-1$ times (see \cite[Section $5.8$]{GeroKoch05}, and \cite{GaoGeroGrynkie08,GaoGeroSchmid07,Wolfgang} for recent progress). It is conjectured that every $n \geq 2$ satisfies Property B.

\begin{theorem}\label{Theoreme Gao 1} 
Conjecture \ref{conjecture Gao} holds whenever:
\begin{itemize}
\item[$(i)$] $G$ is a finite cyclic group.

\item[$(ii)$] $G \simeq C^2_n$, where $n$ satisfies Property B.

\item[$(iii)$] $G \simeq C^r_p$, where $p \in \left\{2,3,5,7\right\}$.
\end{itemize}
\end{theorem}

\begin{theorem}\label{Theoreme Gao 2}
Let $G$ be a finite Abelian group, and let $p$ be the smallest prime divisor of $\exp(G)$. Then Conjecture \ref{conjecture Gao} holds for every integer $i \leq \min\left(6,p-1\right)$. 
\end{theorem}

\section{New results and plan of the paper}
\label{Section New Results}
In this paper, we use the notion of a dispersive sequence to obtain a characterization of short normal sequences over a finite Abelian $p$-group, thus improving both Gao's results on this problem. The main theorem (Theorem \ref{La Conjecture de Gao est vraie pour les p-groupes abeliens}) is proved in Section \ref{Section Main Theorem} yet, before giving this general result, we would like to emphasize its consequences.

\begin{theorem}
\label{Structure des sequences normales d'un p-groupe abelien}
Let $G$ be a finite Abelian $p$-group, and let $S$ be a normal sequence over $G$ with $\left|S\right|=\mathsf{D}(G)+i-1$, where $i \in \llbracket 1,p-1 \rrbracket$. Then $S$ is of the form $S=0^iT$, where $T$ is a zero-sumfree sequence. 
\end{theorem}

Firstly, Theorem \ref{Structure des sequences normales d'un p-groupe abelien} improves Theorem \ref{Theoreme Gao 2}, by showing that the assumption $i \leq 6$ is unnecessary for finite Abelian $p$-groups. Secondly, it improves Statement $(iii)$ of Theorem \ref{Theoreme Gao 1}, by settling Conjecture \ref{conjecture Gao} for every elementary $p$-group. More generally, we can derive immediately from Theorem \ref{Structure des sequences normales d'un p-groupe abelien} the following corollary on Gao's conjecture. 

\begin{corollary}
\label{Structure des sequences normales d'un p-groupe abelien et conjecture de Gao}
Conjecture \ref{conjecture Gao} holds for all groups of the form $G \simeq C_p \oplus H$, where $H$ is any finite Abelian $p$-group.
\end{corollary}

The main result of this paper is deduced from a theorem of Alon, Friedland and Kalai (see \cite[Theorem A.$1$]{Alon84}), originally proved so as to study the existence of regular subgraphs in almost regular graphs. Our result is the following.

\begin{theorem}
\label{La Conjecture de Gao est vraie pour les p-groupes abeliens}
Let $G$ be a finite Abelian $p$-group, and let $S$ be a sequence over $G$ with $\left|S\right|=\mathsf{D}(G)+i-1$, where $i \in \llbracket 1,p \rrbracket$. Let also $\mathcal{A}$ be any $(i-1)$-subset of $\llbracket 1,p-1 \rrbracket$. Then $S$ contains a non-empty zero-sum subsequence $S'$ such that
$$\left|S'\right| \not\equiv b \text{ } (\text{\em mod \em} p), \text{ for all } b  \in \mathcal{A}.$$
\end{theorem}

Consequently, Theorem \ref{La Conjecture de Gao est vraie pour les p-groupes abeliens} gives the existence of non-empty zero-sum subsequences whose length avoids certain remainders modulo $p$. In addition, this result applies to 'short' sequences, the length of which is close to $\mathsf{D}(G)$, thus allowing one to tackle Gao's conjecture, whereas other existing results with a similar flavor (see for instance \cite[Theorem $1.2$]{Schmid01}) hold for longer sequences only. In particular, Theorem \ref{La Conjecture de Gao est vraie pour les p-groupes abeliens} provides the following insight into dispersive sequences.  

\begin{corollary}
\label{Theorem on dispersive sequences}
Let $G$ be a finite Abelian $p$-group, and let $S$ be a sequence over $G$ with $\left|S\right| = \mathsf{D}(G)+i-1$, where $i \geq 1$. The following two statements hold.

\begin{itemize}
\item[$(i)$] If $i \geq 2$ and $S$ contains a zero-sum subsequence $S'$ with $p \nmid\left|S'\right|$, then $S$ is dispersive. 
\item[$(ii)$] If $S$ contains no non-empty zero-sum subsequence $S'$ with $p \mid \left|S'\right|$, then $i \leq p-1$, and $S$ has to contain at least $i$ non-empty zero-sum subsequences with pairwise distinct lengths.
\end{itemize} 
\end{corollary}

In Section \ref{Section p-groupes elementaires}, we give a proof of Theorem \ref{La Conjecture de Gao est vraie pour les p-groupes abeliens}, in the special case of elementary $p$-groups, using the polynomial method. 
Since this proof is short and may be relevant in its own right, we will present it in full. 

\medskip
In Section \ref{Section groupes de rang deux}, we then prove the following theorem, which extends Statement $(ii)$ of Theorem \ref{Theoreme Gao 1} to every finite Abelian group of rank two. The proof of this theorem relies on a structural result obtained by Schmid \cite{Wolfgang}, which is a characterization of long minimal zero-sum sequences over these groups, provided that a suitable divisor of their exponent satisfies Property B. 

\begin{theorem}\label{Groupes de rang deux} Let $G \simeq C_{m} \oplus C_{mn},$ where $m,n \in \mathbb{N}^*$ and $m \geq 2$. Let also $S$ be a normal sequence over $G$ with $\left|S\right|=\mathsf{D}(G)+i-1$, where $i \in \llbracket 1,m-1 \rrbracket$. If $m$ satisfies Property B, then $S$ is of the form $S=0^iT$, where $T$ is a zero-sumfree sequence.  
\end{theorem}

Finally, in Section \ref{Conclusion}, we propose two general conjectures suggested by the results proved in this paper.

\section{The case of finite Abelian $p$-groups}
\label{Section Main Theorem}
As stated in Section \ref{Section New Results}, we prove our Theorem \ref{La Conjecture de Gao est vraie pour les p-groupes abeliens} by using the following theorem of Alon, Friedland and Kalai (see \cite[Theorem A.$1$]{Alon84}). Before stating this result, we need to introduce the following notation. Let $\mathbb{Z}$ be the set of integers. For $\mathcal{S} \subseteq \mathbb{Z}$ and $m \in \mathbb{Z}$, we denote by $\text{card}_m(\mathcal{S})$ the number of distinct elements in $\mathcal{S}$ modulo $m$. 

\begin{theorem}\label{Theoreme Alon}
Let $p$ be a prime, and let $1 \leq d_1 \leq \dots \leq d_n$ be $n$ integers. For $1 \leq j \leq n$, let $\mathcal{S}_j \subseteq \mathbb{Z}$ be a set of integers containing $0$. For $1 \leq i \leq m$, let $(a_{i,1},\dots,a_{i,n})$ be a vector with integer coordinates. If
$$m \geq \displaystyle\sum^n_{j=1}\left(p^{d_j}-\text{\em card\em}_p(\mathcal{S}_j)\right)+1,$$
then a subset $\emptyset \subsetneq I \subseteq \left\{1,\dots,m\right\}$ and numbers $s_j \in \mathcal{S}_j$ $(1 \leq j \leq n)$ exist such that
$$\displaystyle\sum_{i \in I} a_{i,j} \equiv s_j \text{ } (\text{\em mod \em} p^{d_j}), \text{ for all } 1 \leq j \leq n.$$
\end{theorem}

For instance, it may be observed that Theorem \ref{Theoreme Alon} provides the exact value for the Davenport constant of a finite Abelian $p$-group, which was originally obtained by van Emde Boas, Kruyswijk and Olson (see \cite{EmdeBoas67,Olso69a}). 

\medskip
Indeed, let $G \simeq C_{p^{d_1}} \oplus \dots \oplus C_{p^{d_r}}$, where $1 \leq d_1 \leq \dots \leq d_r \in \mathbb{N}$, be a finite Abelian $p$-group, and let us set $\mathsf{D}^*(G)=\sum^r_{i=1} \left(p^{d_i}-1\right)+1.$ On the one hand, an elementary construction (see \cite[Proposition $5.1.8$]{GeroKoch05}) implies that $\mathsf{D}(G) \geq \mathsf{D}^*(G)$. On the other hand, let $(e_1,\dots,e_r)$ be a basis of $G$, where $\text{ord}(e_i)=p^{d_i}$ for all $i \in \llbracket 1,r \rrbracket$, and let $S=g_1 \cdot \hdots \cdot g_m$ be a sequence over $G$ of length $m=\mathsf{D}^*(G)$. Setting $g_i=a_{i,1}e_1 + \dots + a_{i,r}e_r$ for all $i \in \llbracket 1,m \rrbracket$, and $\mathcal{S}_j=\left\{0\right\}$ for every $j \in \llbracket 1,r \rrbracket$, one readily obtains the converse inequality $\mathsf{D}(G) \leq \mathsf{D}^*(G)$ by Theorem \ref{Theoreme Alon}, thus $\mathsf{D}(G) = \mathsf{D}^*(G)$ for every finite Abelian $p$-group. 

\medskip
Using Theorem \ref{Theoreme Alon}, we can now prove the main theorem of this paper.

\begin{proof}[Proof of Theorem \ref{La Conjecture de Gao est vraie pour les p-groupes abeliens}] 
Let $G \simeq C_{p^{d_1}} \oplus \dots \oplus C_{p^{d_r}}$, where $1 \leq d_1 \leq \dots \leq d_r \in \mathbb{N}$, be a finite Abelian $p$-group. Let $(e_1,\dots,e_r)$ be a basis of $G$, where $\text{ord}(e_i)=p^{d_i}$ for all $i \in \llbracket 1,r \rrbracket$, and let $S=g_1 \cdot \hdots \cdot g_m$ be a sequence over $G$ of length $m=\mathsf{D}(G)+i-1$, where $i \in \llbracket 1,p \rrbracket$. 
The elements of $S$ can be written in the following way:
$$\begin{array}{ccccccc}
g_1 & = & a_{1,1}e_1 & + & \dots & + & a_{1,r}e_r,\\
\vdots & & \vdots & & & &  \vdots\\
g_m & = & a_{m,1}e_1 & + & \dots & + & a_{m,r}e_r.
\end{array}$$
Now, let us set $n=r+1$ and $d_n=1$. Let also $\mathcal{A}$ be a $\left(i-1\right)$-subset of $\llbracket 1,p-1 \rrbracket$, and $\bar{\mathcal{A}}=\llbracket 0,p-1 \rrbracket \backslash \mathcal{A}$. For all $j \in \llbracket 1,n \rrbracket$, we set 
$$
\mathcal{S}_j=
\begin{cases}
\text{ } \bar{\mathcal{A}} \hspace{0.23cm} \text{ if } j=n,\\ 
\left\{0\right\} \text{ otherwise.}
\end{cases}
$$
Since $\left(p^{d_n}-\text{card}_p(\mathcal{S}_n)\right)=\left(p-\left|\bar{\mathcal{A}}\right|\right)=\left|\mathcal{A}\right|=i-1$, one obtains
$$\sum^{n}_{j=1}\left(p^{d_j}-\text{card}_p(\mathcal{S}_j)\right) + 1 = \mathsf{D}(G) + i - 1 = m.$$
Therefore, using Theorem \ref{Theoreme Alon}, there exists a subset $\emptyset \subsetneq I \subseteq \left\{1,\dots,m\right\}$ such that
$$
\begin{cases}
\displaystyle\sum_{i \in I} a_{i,j} \equiv 0 \text{ } (\text{mod } p^{d_j}) \text{ for all } 1 \leq j \leq r, \\
\displaystyle\sum_{i \in I} 1 = \left|I\right| \equiv s \text{ } (\text{mod } p) \text{ for some } s \in \bar{\mathcal{A}}.
\end{cases}
$$
Consequently, the sequence $S'=\prod_{i \in I}g_i$ is a non-empty zero-sum subsequence of $S$ such that $\left|S'\right| = \left|I\right| \not\equiv b$ $(\text{mod } p)$ for all $b \in \mathcal{A}$, which is the desired result. 
\end{proof}

We can now prove Theorem \ref{Structure des sequences normales d'un p-groupe abelien} and Corollary \ref{Theorem on dispersive sequences}.
\begin{proof}[Proof of Theorem \ref{Structure des sequences normales d'un p-groupe abelien}] We prove this theorem by induction on $i \in \llbracket 1,p-1 \rrbracket$. If $i=1$, the desired result is straightforward. Now, let $i \geq 2$, and let us assume that the assertion of Theorem \ref{Structure des sequences normales d'un p-groupe abelien} is true for every $1 \leq k \leq i-1$. Let $S$ be a normal sequence over $G$ with $\left|S\right|=\mathsf{D}(G)+i-1$, and let $\mathcal{A}$ be a $(i-1)$-subset satisfying $\left\{i\right\} \subseteq \mathcal{A} \subseteq \llbracket 1,p-1 \rrbracket$. Then, Theorem \ref{La Conjecture de Gao est vraie pour les p-groupes abeliens} gives the existence of a non-empty zero-sum subsequence $S' \mid S$ such that $\left|S'\right| \not\equiv i \text{ } (\text{mod } p)$. In particular, one has $\left|S'\right| \neq i$, and since $S$ is a normal sequence, we must have $\left|S'\right| \leq i-1$. Now, let $T \mid S$ be the sequence such that $S=TS'$. Then $\left|T\right|=\mathsf{D}(G)+ k - 1$, where $1 \leq k = i-\left|S'\right| \leq i-1$. Moreover, since $S$ is a normal sequence, every non-empty zero-sum subsequence $T' \mid T$ has to satisfy $\left|T'\right| \leq k = i-\left|S'\right|$, that is, $T$ is also a normal sequence. Therefore, the induction hypothesis applied to $T$ implies that $S$ is of the form $S=0^kU$. One has $\left|U\right|=\mathsf{D}(G)+ \ell - 1$, where $1 \leq \ell = i-k \leq i-1$, and since $S$ is a normal sequence, every non-empty zero-sum subsequence $U' \mid U$ has to satisfy $\left|U'\right| \leq \ell = i-k$, that is, $U$ is also a normal sequence. Finally, the induction hypothesis applied to $U$ implies that $S$ is of the form $S=0^k0^{i-k}V=0^iV$, where $V$ is a zero-sumfree sequence over $G$, which completes the proof.   
\end{proof}

\begin{proof}[Proof of Corollary \ref{Theorem on dispersive sequences}] $(i)$ Let $S$ be a sequence over $G$ with $\left|S\right| = \mathsf{D}(G)+i-1$, where $i \geq 2$, such that $S$ contains a zero-sum subsequence $S'$ with $p \nmid \left|S'\right|$. Then, one has $\left|S'\right| = qp + r$ for some integers $q \geq 0$ and $r \in \llbracket 1,p-1 \rrbracket$. Now, let $T$ be any subsequence of $S$ such that $\left|T\right| = \mathsf{D}(G) + 2 -1$. Specifying $\mathcal{A}=\{r\}$ in Theorem \ref{La Conjecture de Gao est vraie pour les p-groupes abeliens}, we obtain the existence of a non-empty zero-sum subsequence $T' \mid T$ such that $\left|T'\right| \not\equiv r \text{ } (\text{mod } p)$. In particular, one obtains $\left|T'\right| \not\equiv \left|S'\right| \text{ } (\text{mod } p)$, which implies $\left|T'\right| \neq \left|S'\right|$. Since $T'$ is also a zero-sum subsequence of $S$, the desired result is proved. 

\medskip
\noindent $(ii)$ Let $S$ be a sequence over $G$ with $\left|S\right| = \mathsf{D}(G)+i-1$, where $i \geq 1$, such that $S$  contains no non-empty zero-sum subsequence $S'$ with $p \mid \left|S'\right|$. If one had $i \geq p$, we would obtain, specifying $\mathcal{A}=\llbracket 1,p-1 \rrbracket$ in Theorem \ref{La Conjecture de Gao est vraie pour les p-groupes abeliens}, that every subsequence $T \mid S$ with $\left|T\right| = \mathsf{D}(G)+p-1$ contains a non-empty zero-sum subsequence $T'$ such that $p \mid \left|T'\right|$, which contradicts the assumption made on $S$. Now, we can prove the second part of the assertion, by induction on $k \in \llbracket 1,i \rrbracket$. If $k=1$, then since $\left|S\right| \geq \mathsf{D}(G)$, $S$ has to contain a zero-sum subsequence $S_1$ such that $p \nmid \left|S_1\right|$, and we are done. Now, let $k \in \llbracket 2,i \rrbracket$, and let us assume the assertion is true for $k-1$, that is, $S$ contains at least $k-1$ non-empty zero-sum subsequences $S_1,\dots,S_{k-1}$ with pairwise distinct lengths. By hypothesis, one has $p \nmid \left|S_j\right|$ for all $j \in \llbracket 1,k-1 \rrbracket$, and we can write $\left|S_j\right| = q_jp + r_j$, where $q_j \geq 0$ and $r_j \in \llbracket 1,p-1 \rrbracket$. Now, let $\mathcal{A}$ be a $(i-1)$-subset satisfying $\left\{r_1,\dots,r_{k-1}\right\} \subseteq \mathcal{A} \subseteq \llbracket 1,p-1 \rrbracket$. Then, by Theorem \ref{La Conjecture de Gao est vraie pour les p-groupes abeliens}, we obtain the existence of a non-empty zero-sum subsequence $S_k \mid S$ such that $\left|S_k\right| \not\equiv b \text{ } (\text{mod } p)$ for all $b \in \mathcal{A}$. In particular, $\left|S_k\right| \neq \left|S_j\right|$ for all $j \in \llbracket 1,k-1 \rrbracket$, and consequently, $S$ contains at least $k$ non-empty zero-sum subsequences $S_1,\dots,S_k$ with pairwise distinct lengths, and the proof is complete. 
\end{proof}

\section{The special case of elementary $p$-groups}
\label{Section p-groupes elementaires} 
In this section, we propose an alternative proof of Theorem \ref{La Conjecture de Gao est vraie pour les p-groupes abeliens}, in the special case of elementary $p$-groups, which uses an algebraic tool introduced by Alon and called the \em Combinatorial Nullstellensatz \em (see \cite{Alon99} for a survey on this method). This polynomial method uses the fact that a non-zero multivariate polynomial over a field cannot vanish on 'large' Cartesian products so as to derive a variety of results in combinatorics, additive number theory and graph theory. This method relies on the following theorem. 

\begin{theorem}[Combinatorial Nullstellensatz]
\label{Combinatorial Nullstellensatz} Let $\mathbb{F}$ be a field and let $f$ be a
polynomial in $\mathbb{F}[x_1,\dots,x_n]$ of total degree
$\deg(f)$, admitting a monomial of the following form:
$$x^{\alpha_1}_1 x^{\alpha_2}_2 \cdots x^{\alpha_n}_n \text{ of degree }
\displaystyle\sum^n_{i=1} \alpha_i=\deg(f).$$
Then, for any choice of $n$ subsets $S_1,\dots ,S_n \subseteq
\mathbb{F}$ such that $|S_i|> \alpha_i$ for all $i \in
\llbracket1,n\rrbracket$, there exists an element $(s_1,\dots,s_n)
\in S_1 \times \cdots \times S_n$ such that one has
$f(s_1,\dots,s_n) \neq 0.$
\end{theorem}

\begin{proof}[Proof of Theorem \ref{La Conjecture de Gao est vraie pour les p-groupes abeliens} in the special case of elementary $p$-groups.] Let $p$ be a prime, and let $G \simeq C^r_p$ be an elementary $p$-group of rank $r$. Let also $(e_1,\dots,e_r)$ be a basis of $G$, and let $S=g_1 \cdot \hdots \cdot g_m$ be a sequence over $G$ of length $m=\mathsf{D}(G)+k-1$, where $k \in \llbracket 1,p \rrbracket$. The elements of $S$ can be written in the following way:
$$\begin{array}{ccccccc}
g_1 & = & a_{1,1}e_1 & + & \dots & + & a_{1,r}e_r,\\
\vdots & & \vdots & & & &  \vdots\\
g_m & = & a_{m,1}e_1 & + & \dots & + & a_{m,r}e_r.
\end{array}$$
Let $\mathcal{A}$ be a $(k-1)$-subset of $\llbracket 1,p-1 \rrbracket$, and let also $P \in \mathbb{F}_p \left[x_1,\dots,x_m\right]$ be the following polynomial over the finite field $\mathbb{F}_p$ of order $p$,
$$P(x_1,\dots,x_m)=\displaystyle\prod^r_{h=1}\displaystyle\prod^{p-1}_{j=1}\left(\displaystyle\sum^m_{i=1}a_{i,h}x^{p-1}_i - j\right) \displaystyle\prod_{j \in \mathcal{A}}\left(\displaystyle\sum^{m}_{i=1}x^{p-1}_i -j\right)-\delta\displaystyle\prod^m_{i=1}\left(x^{p-1}_i-1\right),$$
where $\delta \in \mathbb{F}_p$ is chosen such that $P(0,\dots,0)=0$. In particular, since no element of $\mathcal{A}$ is a multiple of $p$, one has $\delta \neq 0$. Moreover, the total degree of $$\displaystyle\prod^r_{h=1}\displaystyle\prod^{p-1}_{j=1}\left(\displaystyle\sum^m_{i=1}a_{i,h}x^{p-1}_i - j\right) \displaystyle\prod_{j \in \mathcal{A}}\left(\displaystyle\sum^{m}_{i=1}x^{p-1}_i -j\right)$$ 
being $\left(r(p-1)+(k-1)\right)\left(p-1\right)=\left(\mathsf{D}(G)+k-2\right)\left(p-1\right)< m(p-1)$, we deduce that $\deg(P)=m(p-1)$. Now, since the coefficient of $\prod^m_{i=1}x^{p-1}_i$ is $-\delta \neq 0$, Theorem \ref{Combinatorial Nullstellensatz} implies that there exists a non-zero element $x=(x_1,\dots,x_m) \in \mathbb{F}^m_p $ such that $P(x_1,\dots,x_m) \neq 0$. Consequently, setting $I=\left\{i \in \llbracket 1,m \rrbracket \text{ } | \text{ } x_i \neq 0 \right\}$, we obtain that $S'=\prod_{i \in I}g_i$ is a non-empty zero-sum subsequence of $S$ satisfying $\left|S'\right| = \left|I\right| \not\equiv b \text{ } (\text{mod } p) \text{ for all } b  \in \mathcal{A}$, which completes the proof. 
\end{proof}

\section{The case of finite Abelian groups of rank two}\label{Section groupes de rang deux}
In this section, we prove a result extending Statement $(ii)$ of Theorem \ref{Theoreme Gao 1} to every finite Abelian group of rank two. The proof of this theorem relies on the following result of Schmid \cite{Wolfgang}, which gives a structural characterization of minimal zero-sum sequences of length $\mathsf{D}(G) = m + mn -1$ over the group $G \simeq C_{m} \oplus C_{mn},$ where $m,n \in \mathbb{N}^*$ and $m \geq 2$, under the hypothesis that $m$ satisfies Property B.
\begin{theorem}\label{Wolfgang} Let $G \simeq C_{m} \oplus C_{mn},$ where $m,n \in \mathbb{N}^*$ and $m \geq 2$, be a finite Abelian group of rank two. The following sequences are minimal zero-sum sequences of maximal length.

\begin{itemize}
\item[$(i)$] 
$S = e^{\text{\em ord\em}(e_j)-1}_j \displaystyle\prod^{\text{\em ord\em}(e_k)}_{i=1} \left(-x_ie_j + e_k\right)$ 
where $(e_1,e_2)$ is a basis of $G$ with $\text{\em ord\em}(e_2) = mn$, $\{j,k\}=\{1,2\}$, and $x_i \in \mathbb{N}$ with $\sum^{\text{\em ord\em}(e_k)}_{i=1} x_i \equiv -1 \text{ } (\text{\em mod \em} \text{\em ord\em}(e_j))$.

\item[$(ii)$] $S = g^{sm-1}_1 \displaystyle\prod^{(n+1-s)m}_{i=1} \left(-x_ig_1 + g_2\right)$ where $s \in \llbracket 1,n \rrbracket$, $\{g_1, g_2\}$ is a generating set of $G$ with $\text{\em ord\em}(g_2) = mn$ and such that $s=1$ or $mg_1 = mg_2$,
and $x_i \in \mathbb{N}$ with $\sum^{(n+1-s)m}_{i=1} x_i = m - 1$.
\end{itemize}

In addition, if $m$ satisfies Property B, then all minimal zero-sum sequences of maximal
length over G are of this form.
\end{theorem}

Moreover, Gao and Zhuang proved a useful structural result on normal sequences (see \cite[Theorem $1.2$]{Gao06}). In this section, we will use the following corollary of this theorem.
\begin{theorem}\label{Gao}
Let $G$ be a finite Abelian group, and let $S$ be a normal sequence over $G$. Then $S$ is of the form $S=0^{k}TU$, where $T$ is a zero-sumfree sequence with $\left|T\right|=\mathsf{D}(G)-1$, and $U$ is a sequence over $G$ such that $\text{\em Supp\em}(U) \subseteq \text{\em Supp\em}(T)$.
\end{theorem}

Using the above two results, we can now prove the following theorem.

\begin{theorem}\label{Longues sous-suites de somme zero} Let $G \simeq C_{m} \oplus C_{mn},$ where $m,n \in \mathbb{N}^*$ and $m \geq 2$, be a finite Abelian group of rank two. Let $T$ be a zero-sumfree sequence over $G$ with $\left|T\right|=\mathsf{D}(G)-1$, and let $U$ be a non-empty sequence over $G$ such that $\text{\em Supp\em}(U) \subseteq \text{\em Supp\em}(T)$. If $m$ satisfies Property B, then every non-empty zero-sum subsequence $S'$ of $S=TU$ has length $\left|S'\right| \geq m$.
\end{theorem}

\begin{proof} [Proof of Theorem \ref{Longues sous-suites de somme zero}] Let $G$ and $S=TU$ be as in the statement of the theorem. By Theorem \ref{Wolfgang}, and since $T'=T(-\sigma(T))$ is a minimal zero-sum sequence over $G$, $S=TU$ can be written in the following fashion
\begin{eqnarray*}
S & = & g^{\ell_1}_1 \displaystyle\prod^{\ell_2}_{i=1} \left(-x_ig_1 + g_2\right),
\end{eqnarray*}
where $\{g_1, g_2\}$ is a generating set of $G$ with $\text{ord}(g_2)=mn$, and $\ell_1,\ell_2 \in \mathbb{N}$ are such that $\ell_1+\ell_2=\left|S\right|$. In particular, one has $mg_1 \in \left\langle mg_2\right\rangle$, and $ag_1 \in \left\langle g_2\right\rangle$ if and only if $m \mid a$. Since $\left|S\right| \geq \mathsf{D}(G)$, it has to contain a non-empty zero-sum subsequence $S'$. Now, let us write $S'=VW$ where $V \mid g^{\ell_1}_1$ and $W=\prod_{i \in I} \left(-x_ig_1 + g_2\right)$, for some $I \subseteq \llbracket 1,\ell_2 \rrbracket$. We obtain
$$\sigma(S')=\left|V\right|g_1 - \left(\displaystyle\sum_{i \in I}x_i\right)g_1 + \left|W\right|g_2=0.$$
Since $\sigma(S')=0 \in \left\langle g_2 \right\rangle$, we have 
$$m \mid \left|V\right|-\left(\displaystyle\sum_{i \in I}x_i\right)$$ 
which implies that, for some $a \in \mathbb{N}$, one has 
$$\left|V\right|-\left(\displaystyle\sum_{i \in I}x_i\right)=amg_2.$$
Thus, $\sigma(S')=\left(am+\left|W\right|\right)g_2=0$, which gives $m \mid \text{ord}(g_2) \mid am + \left|W\right|$. Consequently, $m \mid \left|W\right|$. If $\left|W\right| \geq m$, then we are done. Otherwise, one has $\left|W\right|=\left|I\right|=0$, and $\sigma(S') = \left|V\right|g_1= 0 \in \left\langle g_2\right\rangle$. Thus, $m \mid \left|V\right|=\left|S'\right|$, and since $S'$ is a non-empty zero-sum subsequence of $S$, we obtain $\left|S'\right| \geq m$, which is the desired result.
\end{proof}

Theorem \ref{Groupes de rang deux} is now an easy corollary of Theorem \ref{Longues sous-suites de somme zero}.
\begin{proof}[Proof of Theorem \ref{Groupes de rang deux}] Let $G \simeq C_{m} \oplus C_{mn},$ where $m,n \in \mathbb{N}^*$ and $m \geq 2$, be a finite Abelian group of rank two. Let also $S$ be a normal sequence over $G$ with $\left|S\right|=\mathsf{D}(G)+i-1$, where $i \in \llbracket 1,m-1 \rrbracket$. By Theorem \ref{Gao}, $S$ is of the form $S=0^{k}TU$, where $T$ is a zero-sumfree sequence with $\left|T\right|=\mathsf{D}(G)-1$, and $U$ is a sequence over $G$ such that $\text{Supp}(U) \subseteq \text{Supp}(T)$. Since $m >i$, $S$ does not contain any zero-sum subsequence $S'$ with $\left|S'\right| \geq m$. So, it follows from Theorem \ref{Longues sous-suites de somme zero} that $U$ is empty, which implies $k=i$ and completes the proof.
\end{proof}

\vspace{-0.38cm}
\section{Two concluding remarks}
\label{Conclusion}
In this section, we would like to present two conjectures suggested by the results proved in this paper. The first one may be interpreted as a more general version of Theorem \ref{La Conjecture de Gao est vraie pour les p-groupes abeliens}, and would imply Conjecture \ref{conjecture Gao} in the same way as Theorem \ref{La Conjecture de Gao est vraie pour les p-groupes abeliens} implies Theorem \ref{Structure des sequences normales d'un p-groupe abelien}.

\begin{conjecture}
\label{conjecture generale Girard} 
Let $G \simeq C_{n_1} \oplus \dots \oplus C_{n_r},$ with $1 < n_1 \text{ } |\text{ } \dots \text{ }|\text{ } n_r \in \mathbb{N}$, be a finite Abelian group, and let $S$ be a sequence over $G$ with $\left|S\right|=\mathsf{D}(G)+i-1$, where $i \in \llbracket 1,n_1 \rrbracket$. Let also $\mathcal{A}$ be any $(i-1)$-subset of $\llbracket 1,n_1-1 \rrbracket$. Then $S$ contains a non-empty zero-sum subsequence $S'$ such that
$$\left|S'\right| \not\equiv b \text{ } (\text{\em mod \em} n_1), \text{ for all } b  \in \mathcal{A}.$$
\end{conjecture}

\medskip
Let $G$ be a finite Abelian group of exponent $m$. By $\mathsf{s}_{m\mathbb{N}}(G)$ we denote the smallest $t \in \mathbb{N}^*$ such that every sequence $S$ over $G$ with $|S| \geq t$ contains a non-empty zero-sum subsequence $S'$ with $\left|S'\right| \equiv 0 \text{ } (\text{mod } m).$ The exact value of the invariant $\mathsf{s}_{m\mathbb{N}}(G)$ is currently known for finite Abelian groups of rank $r \leq 2$, and finite Abelian $p$-groups only (see \cite[Theorem $6.7$]{GaoGero06}). For instance, Conjecture \ref{conjecture generale Girard} implies that for all integers $n,r \geq 1$, one has $\mathsf{s}_{n\mathbb{N}}(C^r_n)=(r+1)(n-1)+1$. This conjecture would also help to tackle the inverse problem associated to $\mathsf{s}_{m\mathbb{N}}(G)$, by giving an account of the variety of zero-sum subsequences contained in a long sequence without any non-empty zero-sum subsequence of length congruent to $0$ modulo $m$. Finally, any progress on this conjecture would provide a new insight into the structure of sequences over $C_n$ without any zero-sum subsequence of length $n$ (see \cite[Theorem $7.5$ and Conjecture $7.6$]{GaoGero06}, as well as \cite{SavChen07}).

\medskip
Let $G$ be a finite Abelian group of exponent $m$, and let $\ell \geq 1$ be an integer. By $\eta_{\ell m}(G)$ we denote the smallest $t \in \mathbb{N}^*$ such that every sequence $S$ over $G$ with $|S| \geq t$ contains a non-empty zero-sum subsequence $S'$ with $\left|S'\right| \leq \ell m$. It may be observed that, for every $\ell \geq \left\lceil \mathsf{D}(G)/m\right\rceil$, one has the equality $\eta_{\ell m}(G)=\mathsf{D}(G)$. Now, our second conjecture is the following.

\begin{conjecture}
\label{conjecture eta}
Let $G$ be a finite Abelian group of exponent $m$, and let $S$ be a sequence over $G$ with $\left|S\right|=\eta_{\ell m}(G)+i-1$, where $i \in \llbracket 1,\ell m \rrbracket$. Then $S$ contains a zero-sum subsequence $S'$ of length $i \leq \left|S'\right| \leq \ell m$. 
\end{conjecture}

\medskip
Let $G$ be a finite Abelian group of order $n$ and exponent $m$. For instance, if $\ell=1$ in Conjecture \ref{conjecture eta}, one obtains a generalization of Conjecture $6.5$ in \cite{GaoGero06}, due to Gao. If $\ell=n/m$, then, since it is known that $\mathsf{D}(G) \leq n=\ell m$, we obtain $\eta_{\ell m}(G)=\mathsf{D}(G)$. Therefore, Conjecture \ref{conjecture eta} implies that every sequence $S$ over $G$ such that $\left|S\right|=\mathsf{D}(G)+i-1$, where $i \in \llbracket 1,n \rrbracket$, has to contain a zero-sum subsequence $S'$ with $i \leq \left|S'\right| \leq n$, which can be seen as a generalization of Gao's theorem (see \cite[Theorem $1$]{Gao96}). Finally, if $\ell=\left\lceil \mathsf{D}(G)/m\right\rceil$ and $i = \ell m$, then Conjecture \ref{conjecture eta} implies that every sequence $S$ over $G$ with $\left|S\right|=\mathsf{D}(G)+ \ell m - 1$ has to contain a zero-sum subsequence $S'$ with $\left|S'\right| = \ell m$, which would provide an answer to a problem of Gao (see \cite[Section $3$]{Gao03} and \cite[Theorem $6.12$]{GaoGero06}).

\vspace{-0.3cm}
\section*{Note added in proof}
Shortly after this paper was accepted for publication, Reiher \cite{Reiher} proved that Property B holds for all primes. 
It follows from the results of \cite{GaoGeroGrynkie08} that all integers $n \geq 2$ satisfy Property B. 
Thus, our Theorem \ref{Groupes de rang deux} holds unconditionally, which gives a positive answer to Conjecture \ref{conjecture Gao} for all Abelian groups of rank two. 
Further progress has since been made on this conjecture \cite{GuanYuanZeng}.

\vspace{-0.3cm}
\section*{Acknowledgments}
This work was supported by a postdoctoral grant awarded by the \'{E}cole polytechnique, Paris. I am also grateful to Melvyn Nathanson and to the City University of New York's Graduate Center for all their hospitality and for providing an excellent atmosphere for research.


\end{document}